\documentclass[12pt]{amsart}

\usepackage{amssymb}
\usepackage{amsmath}
\usepackage{amsthm}
\usepackage{amscd}
\usepackage{enumerate}
\usepackage[margin=1.1in]{geometry}
\usepackage{hyperref}
\usepackage{appendix}
\usepackage{dsfont,txfonts}
\usepackage{tikz}
\usetikzlibrary{calc,positioning,fit,backgrounds}

\theoremstyle{plain}
\newtheorem{theorem}{Theorem}[section]
\newtheorem{proposition}[theorem]{Proposition}
\newtheorem{corollary}[theorem]{Corollary}
\newtheorem{lemma}[theorem]{Lemma}

\theoremstyle{remark}

\newtheorem{remark}{Remark}[section]

\DeclareMathOperator{\dom}{dom}

\newcommand{\SG}{\ensuremath{\mathcal{S}}}
\newcommand{\VLN}{\ensuremath{\mathcal{V}_{N,L}}}
\newcommand{\DF}{\mathcal{E}}

\newcommand{\op}{\mathcal{L}}
\newcommand{\degr}{2}
\newcommand{\Wsp}{\ensuremath{\op^{-s}L^{p}}}

\begin{document}

\title{Sobolev algebra counterexamples}

\author{Thierry Coulhon}
\address{PSL Research University, F-75005 Paris, France}
\email{thierry.coulhon@univ-psl.fr}
\thanks{Research of T.C. supported in part by Australian Research Council grant DP 130101302}

\author{Luke~G. Rogers}
\address{Department of Mathematics, University of Connecticut, Storrs, CT 06269-3009 USA}
\email{rogers@math.uconn.edu}

\date{today}

\begin{abstract}
In the Euclidean setting the Sobolev spaces $W^{\alpha,p}\cap L^\infty$  are algebras for the pointwise product when $\alpha>0$ and $p\in(1,\infty)$.  This property has recently been extended to a variety of geometric settings.  We produce a class of fractal examples where it fails for a wide range of the indices $\alpha,p$.

\tableofcontents
\end{abstract}
\maketitle

\section{Introduction}

We consider a  measure space $(X,\mu)$ equipped with a  non-negative definite, self-adjoint, Markovian operator $\op$ with dense domain in $L^{2}(\mu)$.  
Such operators play the role of the classical Laplacian when studying physical phenomena such as diffusion and waves (e.g. the heat, wave, and Schr\"{o}dinger equations)  and related PDE  on a general space $(X,\mu)$.
The natural setting for such problems is a class of Sobolev spaces associated to $\op$; following the correspondence from the Euclidean setting we define these as Bessel potential spaces, so that the homogeneous Sobolev space $\dot{W}_{\op}^{\alpha,p}$ and the inhomogeneous Sobolev space $W_{\op}^{\alpha,p}$ are as follows
\begin{gather}
	\dot{W}_{\op}^{\alpha,p} = \Bigl\{ f\in L^{p}_{\text{loc}}(X,\mu): \op^{\alpha/\degr}f\in L^{p}(X,\mu)\Bigr\} \text{ with seminorm } \|f\|_{\dot{W}_{\op}^{\alpha,p}} = \| \op^{\alpha/\degr}f\|_{p} \\
	W_{\op}^{\alpha,p} = \Bigl\{ f\in L^{p}(X,\mu): \op^{\alpha/\degr}f\in L^{p}(X,\mu)\Bigr\} \text{ with norm } \|f\|_{W_{\op}^{\alpha,p}} = \|f\|_{p} + \| \op^{\alpha/\degr}f\|_{p} \label{eqn:defnWsp}.
	\end{gather}
The {\em Sobolev algebra problem} asks for conditions under which the spaces $\dot{W}_{\op}^{\alpha,p}\cap L^{\infty}$ or $W_{\op}^{\alpha,p}\cap L^{\infty}$ are algebras under the pointwise product.  This question arises when considering the well-posedness of non-linear PDE based on the differential operator $\op$.   The purpose of this paper is to show that on some fractal spaces  we may take $\op$ to be a natural Laplacian operator and nonetheless the algebra property fails for a wide range of $p$ and $\alpha$. The Sobolev spaces we consider have previously been studied in~\cite{HuZ1,HuZ2}.  Our results are close kin to a result of Ben-Bassat, Strichartz and Teplyaev~\cite{BST} which applies (essentially) to the case $p=\infty$, $\alpha=\degr$, and are in sharp contrast to the behavior of the classical Sobolev spaces on Euclidean spaces.  Indeed, in the case that $\op$ is the non-negative Laplacian $-\Delta$ on $\mathbb{R}^{n}$, Strichartz~\cite{Strich} proved that the classical Bessel potential space $W_{-\Delta}^{\alpha,p}$ is an algebra provided $1<p<\infty$, $\alpha>0$ and $\alpha p>n$.  More generally, Kato and Ponce~\cite{KP} showed $W_{\Delta}^{\alpha,p}\cap L^{\infty}$ is an algebra assuming only $1<p<\infty$ and $\alpha>0$. In the homogeneous case $\dot{W}_{-\Delta}^{\alpha,p}\cap L^{\infty}$  was proved to be an algebra for the same range of $p$ and $\alpha$  by Gulisashvili and Kon~\cite{GK}.

Outside the Euclidean setting there are positive results due to Coulhon, Russ and Tardivel-Nachef~\cite{CRT} on Lie groups with polynomial volume growth and on Riemannian manifolds with positive injectivity radius and non-negative Ricci curvature.  Results under weaker geometric conditions were later obtained by Badr, Bernicot and Russ~\cite{BBR} and most recently by Bernicot, Coulhon and Frey~\cite{BCF}.  There are two main approaches: one is to characterize when $f$ is in the Sobolev space using functionals defined from suitable averaged differences of $f$ and the other is to take a paraproduct decomposition of the product and use square function estimates to reduce the problem to the Leibniz property of a gradient operator associated to $\op$.  Since our interest in this paper will be in negative results we will not attempt to describe the precise state of the art but instead isolate two theorems which give positive results of a similar type.  It should be emphasized that these results were chosen for the simplicity of their statements, and are far from the most general statements proved in~\cite{CRT,BCF}.

\begin{theorem}[\protect{\cite{CRT} Theorem~2}]
Let $G$ be a connected Lie group of polynomial volume growth, equipped with Haar measure and a family $Y_{j}$ of left-invariant H\"{o}rmander vector fields.  Let $\op$ be the associated sub-Laplacian $-\sum_{j=1}^{k}Y_{j}^{2}$.  For $\alpha\geq0$ and $1<p<\infty$ the space $\dot{W}_{\op}^{\alpha,p}\cap L^{\infty}$ is an algebra under the pointwise product.
\end{theorem}

\begin{theorem}[\protect{\cite{BCF} Theorem~1.5}]\label{thm:BCF}
Let $(M,d,\mu,\DF)$ be a doubling metric measure space with Dirichlet form $\DF$ and associated operator $\op$.  Suppose the energy measures of functions in the domain of $\DF$ (in the sense of Beurling-Deny) are absolutely continuous with respect to $\mu$, and that for all $x\in M$ balls are measure doubling with $\mu(B(x,r_1))\leq (r_{1}/r_{2})^{\nu}B(x,r_{2})$ for $0<r_{2}\leq r_{1}$ and some $\nu>0$.  Further assume that the heat semigroup associated to $\op$ has a kernel $h_{t}$ satisfying $h_{t}(x,y)\leq \bigl( \mu(B(x,\sqrt{t})\mu(B(y,\sqrt{t})\bigr)^{-1/2}$ for a.e.~$x,y\in M$ and $t>0$.  Then for $p\in(1,2]$ and $0<\alpha<1$ or for $p\in(2,\infty)$ and $0<\alpha<1-\nu\bigl(\frac{1}{2}-\frac{1}{p}\bigr)$ we have that $\dot{W}_{\op}^{\alpha,p}\cap L^{\infty}$ is an algebra.  (See Figure~\ref{fig:BCF} for an illustration of the $(\alpha,p)$ region in which the corresponding Sobolev spaces have the algebra property.) 
\end{theorem}

\begin{figure}
\begin{tikzpicture}
\path[fill=gray] (1,0)--(4,2) --(4,4)--(0,4)--(0,0)--(1,0);
\draw[thick, <->] (0,4.5) -- (0,0) -- (5,0);
\draw (1,-.1)--(1,.1);
\draw (4,-.1)--(4,.1);
\draw (-.1,2)--(.1,2);
\draw (-.1,4)--(.1,4);
\draw[thick, dashed] (0,4) -- (4,4);
\draw[thick, dashed] (1,0) -- (4,2);
\draw[thick, dashed] (4,2) -- (4,4);
\node at (1,-0.5) {$1-\frac{\nu}{2}$};
\node at (4,-0.5) {$1$};
\node at (5.5,0) {$\alpha$};
\node at (-0.4,4) {$1$};
\node at (-0.4,2) {$\frac{1}{2}$};
\node at (-1,2.25) {$\displaystyle\frac{1}{p}$};
\end{tikzpicture}
\caption{For $\op$ satisfying the assumptions of \protect{Theorem~\ref{thm:BCF}} and $(\alpha,p)$ in the shaded region, $\dot{W}^{\alpha,p}\cap L^{\infty}$ is an algebra.}\label{fig:BCF}
\end{figure}
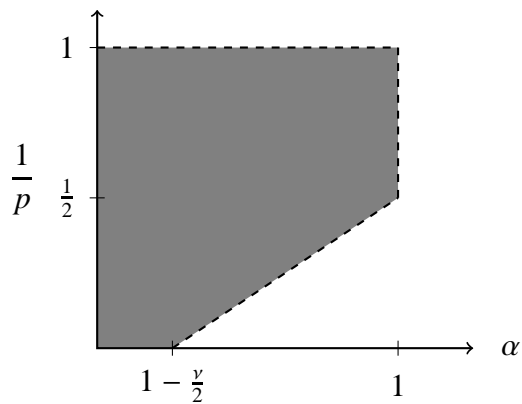

The  results of the present work indicate some of the obstacles that may be encountered in extending Theorem~\ref{thm:BCF} to larger $\alpha$.  It should be noted that several of the hypotheses of Theorem~\ref{thm:BCF} fail in the examples of Theorem~\ref{thm:mainresult}.  In our examples the energy measures of functions in the domain of $\DF$ are not absolutely continuous with respect to the measure $\mu$; the significance of this for failure of the algebra property was a feature of one of the basic arguments of~\cite{BST}.  At the same time, the upper estimate $h_{t}(x,y)\leq \bigl(\mu(B(x,\sqrt{t}))\mu(B(y,\sqrt{t})\bigr)^{-1/2}$ is also invalid on our examples, instead being replaced by $h_{t}(x,y)\leq C \bigl(\mu(B(x,t^{1/\beta}))\mu(B(y,t^{1/\beta})\bigr)^{-1/2}$ for $0<t<1$ and constants $C>0$ and $\beta>2$.  The exponent $\beta$ here is the so-called walk dimension of the diffusion with generator $\op$.  In our examples $\beta=D+1$, where $D>1$ is the Hausdorff dimension of the space $X$.  Our first result illustrates the fact that the algebra property can fail for a wide range of indices.

\begin{theorem}\label{thm:mainresult}
Given $\alpha\in(1,2)$ and $p\in(1,\infty]$ satisfying $\alpha p>2$ there is a compact metric space $X$ with Ahlfors regular measure $\mu$ and a Laplacian operator $\op$ which is densely-defined on $L^{2}(\mu)$, non-positive definite, self-adjoint, Markovian,  strongly local, and such that neither $W^{\alpha,p}_\op$ nor $\dot{W}^{\alpha,p}_\op\cap L^{\infty}$ is an algebra. Figure~\ref{fig:mainresult} illustrates the corresponding region of $\alpha$ and $p$ values.
\end{theorem}
\begin{proof}
This follows directly from Theorem~\ref{thm:VNLexamplesDlimit} and Corollary~\ref{cor:bddness} below.
\end{proof}

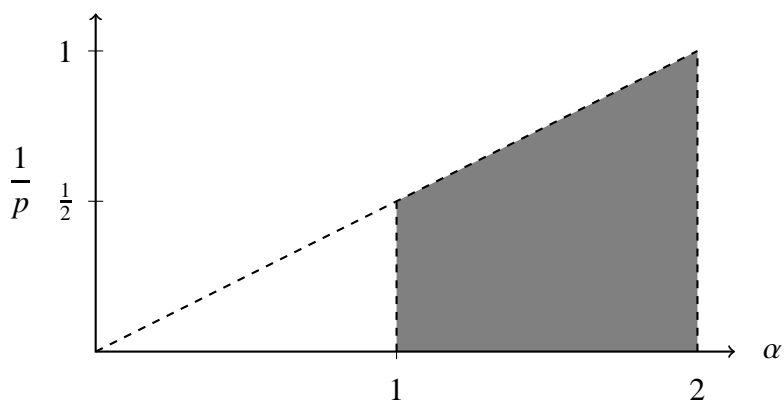
\begin{figure}
\begin{tikzpicture}
\path[fill=gray] (4,0)--(4,2) --(8,4)--(8,0)--(4,0);
\draw[thick, <->] (0,4.5) -- (0,0) -- (8.5,0);
\draw (4,-.1)--(4,.1);
\draw (-.1,4)--(.1,4);
\draw (-.1,2)--(.1,2);
\draw[thick, dashed] (0,0) -- (8,4);
\draw[thick, dashed] (4,0) -- (4,2);
\draw[thick, dashed] (8,0) -- (8,4);
\node at (4,-0.5) {$1$};
\node at (9,0) {$\alpha$};
\node at (-0.4,4) {$1$};
\node at (8,-0.5) {$2$};
\node at (-1,2.25) {$\displaystyle\frac{1}{p}$};
\node at (-0.4,2) {$\frac{1}{2}$};
\end{tikzpicture}
\caption{The $(\alpha,p)$ values for which  \protect{Theorem~\ref{thm:mainresult}} proves $W^{\alpha,p}$ and  $\dot{W}^{\alpha,p}\cap L^{\infty}$ can fail to be an algebras.}\label{fig:mainresult}
\end{figure}

The paper is organized as follows.  Section~\ref{sec:Fractalsandresolvent} gives some standard background and assumptions for our class of fractal examples.  In Section~\ref{Section:localbehavior} we use heat kernel estimates and additional features of the fractal structure to analyze the local behavior of Sobolev functions. The failure of the algebra property is discussed in Section~\ref{Section:algebraprop}; as a consequence of our discussion we also note that the Sobolev space fails to be preserved by the action of a function that is differentiable and has a certain lower convexity bound, see Corollary~\ref{cor:noconvexfns}.  Finally, in Section~\ref{Section:Examples}, we present some specific collections of fractal spaces on which our assumptions hold and complete the arguments that prove Theorem~\ref{thm:mainresult}.

\section{Laplacians on a class of fractals and estimates for the resolvent}\label{sec:Fractalsandresolvent}
We consider a post-critically finite fixed set $X\subset \mathbb{R}^{N}$ of an iterated function system $\{F_{j}\}_{j=1,\dotsc J}$.  For detailed definitions of such sets and their properties, including the definition and all results on resistance forms stated without proof below, see~\cite{Kigamibook}.  We write $V_{0}$ for the finite post-critical set, which we consider to be the boundary of $X$.   For a word $w=w_{1}\dotsm w_{m}$ of length $|w|=m$ with letters from $\{1,\dotsc,J\}$ let $F_{w}=F_{w_{1}}\circ\dotsm\circ F_{w_{m}}$.  We refer to $F_{w}(X)$ with $|w|=m$ as an $m$-cell of $X$.  Define $V_{m}=\cup_{|w|=m}F_{w}(V_{0})$ and consider these to be the vertices of a graph in which the edge relation $x\sim_{m}y$ means that $x,y\in V_{m}$ and there is $w$ with $|w|=m$ and both $x,y\in F_{w}(V_{0})$.  On this $m$-scale graph there is a resistance form $\DF_{m}=\sum_{x\sim_{m}y}(u(x)-u(y))^{2}$. 

We assume there is a resistance renormalization constant $0<r<1$ such that $\lim_{m\to\infty} r^{-m}\DF_{m}(u)$ is non-decreasing with limit $\DF(u)$, and this defines a regular resistance form on $V_{\ast}=\cup_{m}V_{m}$ with domain the set $\dom(\DF)=\{u:\DF(u)<\infty\}$. Note that the functions with $\DF(u)=0$ are constants. The resistance form is self-similar in that $\DF(u\circ F_{w}^{-1})=r^{|w|}\DF(u)$ for any $u$ on $F_{w}(V_{\ast})$ so that $u\circ F_{w}^{-1}\in\dom(\DF)$, and it defines the resistance metric $R(x,y)=\sup\{\DF(u)^{-1}:u(x)=0,u(y)=1\}$.  Functions in $\dom(\DF)$ are then $\frac{1}{2}$-H\"{o}lder with respect to $R(x,y)$ because $|u(x)-u(y)|^{2}\leq \DF(u)R(x,y)$, so they extend from $V_{\ast}$ to its $R$-completion, which is $X$. Given a function on a finite subset $Y\subset X$ there is an element of $\dom(\DF)$ which extends the function on $Y$ to $X$; among such extensions there is a unique minimizer of $\DF$, and such minimizers are said to be harmonic on $X\setminus Y$.  If $Y=V_{0}$ the minimizers are simply called harmonic functions, and if $Y=V_{m}$ they are called piecewise harmonic of scale $m$.  For convenience we scale $\DF$ so that the $R$-diameter of the space is $1$.

In addition to its resistance structure we equip $X$ with the unique self-similar measure in which all $m$-cells have equal mass $\mu^{m}$ for some constant $\mu\in(0,1)$. Then $\mu=r^{D}$ for $D$ the Minkowski dimension of $X$ in the resistance metric;  in all cases we consider $D$ is also the Hausdorff dimension by a well-known result of Hutchinson~\cite{H}.  Abusing notation we also use $\mu$ to denote the measure.  From a theorem of Kigami (see Chapter 3 of~\cite{Kigamibook}) $\DF$ is a Dirichlet form on $L^{2}(\mu)$, whence by standard results (e.g., from Chapter~1 of~\cite{FOT}) we may define a non-negative definite self-adjoint (Dirichlet) Laplacian by setting
\begin{equation}\label{eqn:defnofDelta}
	\DF(u,v) = \int (\op u) v\, d\mu \text{ for all $v\in\dom_{0}(\DF)$}
	\end{equation}
where $\dom_{0}(\DF)$ denotes the functions in the domain of $\DF$ that vanish on $V_{0}$. (Note that it is more usual in the fractal literature to define a non-positive definite Laplacian $\Delta$; for us $\op=-\Delta$.) This Laplacian has compact resolvent and therefore its spectrum consists of non-negative eigenvalues accumulating only at $\infty$.  Moreover the eigenvalue of least magnitude is $\lambda_{1}>0$.
One may also define a Neumann Laplacian by instead requiring that~\eqref{eqn:defnofDelta} holds for all $v\in\dom(\DF)$; our results are unchanged if the Neumann Laplacian is used in place of the Dirichlet Laplacian. We define a normal derivative $du(q)$ at $q\in V_{0}$ by $du(q)=\lim_{m\to\infty}r^{-m}\sum_{x\sim_{m}q}(u(q)-u(x))$. This exists for all sufficiently regular $u$ (for precise conditions see~\cite{KigamiJFA03}), and in particular once $\op u$ exists as a measure (in the sense that $v\mapsto\DF(u,v)$ is a bounded linear functional on $\dom_{0}(\DF)$ with respect to the uniform norm).  Then $\DF(u,v)=\int(\op u)v\,d\mu+\sum_{q\in V_{0}}(du(q))v(q)$ for all $v\in\dom(\DF)$.

The normal derivative may be localized to a boundary vertex $F_{w}(p)$ of the cell $F_{w}(X)$ by setting
\begin{equation}\label{eqn:d}
	du(F_{w}(p))=\lim_{m\to\infty}r^{-m} \sum_{\substack{x\sim_{m}F_{w}(p)\\x\in F_{w}(X)} } \bigl(u(F_{w}(p)) - u(x) \bigr).
	\end{equation}
in which case we obtain a local Gauss-Green formula
\begin{equation}\label{eqn:locGG}
	\int_{F_{w}(X)} (\op u) v-  u(\op v) \, d\mu
	= \sum_{p\in F_{w}(V_{0})} u(p)(dv(p))- (du(p))v(p)
	\end{equation}

We make a strong assumption on the resistance metric and the heat semigroup associated to our Dirichlet form.  Specifically we assume there is $0<\gamma<D+1$ such that $R(x,y)^{\gamma}$ is comparable to a metric on $X$, and a heat kernel $h_{t}(x,y)$ for $e^{t\op}$ satisfying
\begin{equation}\label{eqn:HKests}
	h_{t}(x,y)
	\leq C_{H} t^{-D/(D+1)} \exp \Bigl( - c_{H}\Bigl(\frac{R(x,y)^{(D+1)}}{t} \Bigr)^{\gamma/(D+1-\gamma)} \Bigr)\quad\text{for $0<t<1$.}
	\end{equation}
For $t\geq1$ we may of course use the estimate from the spectral gap: $h_{t}(x,y)\leq e^{-\lambda_{1}t}$.  Although these assumptions seem very restrictive, they are known to be true for a large class of fractals that includes the examples in Section~\ref{Section:Examples}.   In particular both~\eqref{eqn:HKests} and a lower bound of the same type were proved for affine nested fractals in~\cite{FHK}.  Henceforth for notational convenience we write $A\lesssim B$ if  $A/B$ is bounded by constant depending only on the fractal and its Dirichlet form.  Implicitly, then, $A\lesssim B$ involves a constant that may depend on $r,\mu,\gamma, \lambda_{1}$ and the above constants $C_{H}$ and $c_{H}$.

The heat kernel bounds will be used to obtain regularity estimates for the various kernels we use to analyze the local properties of Sobolev functions. In practice we will work primarily with the kernel $G_{\lambda}(x,y)$ of the resolvent $(\lambda+\op)^{-1}$ with $\lambda>0$, which may be obtained via $G_{\lambda}=\int_{0}^{\infty} e^{-\lambda t}h_{t}\, dt$, and with the Riesz kernel $K_{s}(x,y)$ of $\op^{-s}$ for $s\in(0,1)$, which may be obtained as $C_{s}\int_{0}^{\infty} t^{s-1} h_{t}\,dt$ or as $C'_{s}\int_{0}^{\infty} \lambda^{-s} G_{\lambda}(x,y)\,d\lambda$ for suitable constants $C_{s},C'_{s}$ which will henceforth be suppressed.  Inevitably we will frequently need bounds of the following type
\begin{gather}
	\int_{0}^{A} \tau^{a} \exp(-\kappa \tau^{b}) \, \frac{d\tau}{\tau}
	=\begin{cases}
	\kappa^{-a/b} \int_{0}^{\kappa A^{b}}  u^{a/b}  e^{-u} \frac{du}{bu} \leq C_{a,b}\kappa^{-a/b} &\text{ if $a>0$, $b>0$}\\
	\kappa^{-a/b} \int_{\kappa A^{b}}^{\infty}  u^{a/b}  e^{-u} \frac{du}{|b|u} \leq C_{a,b} A^{a} \exp(-\kappa A^{b}) &\text{ if $b<0$}
	\end{cases}\label{eqn:integralest0A}\\
	\int_{A}^{\infty} \tau^{a} \exp(-\kappa \tau^{b}) \, \frac{d\tau}{\tau}
	=\begin{cases}
	\kappa^{-a/b} \int_{\kappa A^{b}}^{\infty}  u^{a/b}  e^{-u} \frac{du}{bu} \leq  C_{a,b} A^{a} \exp(-\kappa A^{b}) &\text{ if $b>0$}\\
	\kappa^{-a/b} \int_{0}^{\kappa A^{b}}  u^{a/b}  e^{-u} \frac{du}{|b|u} \leq  C_{a,b}\kappa^{-a/b} &\text{ if $a<0$, $b<0$}
	\end{cases}\label{eqn:integralestAinfty}
\end{gather}

For example we have
\begin{proposition}\label{prop:Rests} There is a constant $c>0$ so that for $\lambda\geq 0$,
\begin{equation}\label{eqn:Rests}
	\bigl| G_{\lambda}(x,y) \bigr|
	\lesssim 
	(1+\lambda)^{-1/(D+1)} \exp \Bigl( - c  R(x,y)^{\gamma} \lambda^{\gamma/(D+1)} \Bigr).
\end{equation}
\end{proposition}
\begin{proof}
Write the resolvent kernel as $G_{\lambda}(x,y) =\int e^{-\lambda t}h_{t}(x,y)\, dt$ and split the integral over $[0,A)$ and $[A,\infty)$ for $A=  R(x,y)^{\gamma}(1+\lambda)^{-(D+1-\gamma)/(D+1)}\leq1$. On $[0,A]$ bound $e^{-\lambda t}h_{t}$ by the heat kernel estimate~\eqref{eqn:HKests} obtaining from~\eqref{eqn:integralest0A} that
\begin{align*}
	 \int_{0}^{A} t^{1/(D+1)} \exp \Bigl( - c_{H}\Bigl(\frac{R(x,y)^{(D+1)}}{t} \Bigr)^{\gamma/(D+1-\gamma)} \Bigr)\, \frac{dt}{t}
	&\lesssim A^{1/(D+1)} \exp \Bigl( - c_{H} \Bigl( \frac{R(x,y)^{(D+1)}}{A}\Bigr)^{\frac{\gamma}{(D+1-\gamma)} } \Bigr) \\
	&\lesssim(1+\lambda)^{-1/(D+1)} ((1+\lambda)A)^{1/(D+1)} \exp \bigl( - c_{H} (1+\lambda) A \bigr)\\
	&\lesssim  (1+\lambda)^{-1/(D+1)} \exp \Bigl( - c' R(x,y)^{\gamma} (1+\lambda)^{\gamma/(D+1)} \Bigr) \label{eqn:step1inRest}
	\end{align*}
where in the last step we used that $a^{1/(D+1)}e^{-c_{H}a}$ is bounded by a multiple of $e^{-c'a}$ for some choice of $c'<c_{H}$.

On $[A,1)$ we can bound the integrand  by $t^{-D/(D+1)}e^{-\lambda t}\leq et^{-D/(D+1)}e^{-(1+\lambda) t}$, with the power of $t$ coming from~\eqref{eqn:HKests}. A similar bound holds on $[1,\infty)$ because the spectral gap implies $h_{t}(x,y)\lesssim e^{-\lambda_{1}t}\leq C_{\lambda_{1},D}t^{-D(D+1)}e^{-|\lambda_{1}|/2t}$ and $e^{-(\lambda+|\lambda_{1}|/2)t}\leq e^{-c'' (1+\lambda)t}$ for some $c''=c''(\lambda_{1})\leq1$.  Thus using~\eqref{eqn:integralestAinfty} the contribution to the resolvent does not exceed
\begin{align*}
	C (1+\lambda)^{-1/(D+1)} \int_{(1+\lambda) A}^{\infty} u^{1/(D+1)}e^{-u} \frac{du}{u}
	&\lesssim (1+\lambda)^{-1/(D+1)}\bigl((1+\lambda) A\bigr)^{1/(D+1)} e^{-c''(1+\lambda) A} \\
	&\lesssim (1+\lambda)^{-1/(D+1)} \exp \Bigl( - c''' R(x,y)^{\gamma} (1+\lambda)^{\gamma/(D+1)} \Bigr)
	\end{align*}
for a suitable $c$ depending on $c'(\lambda_{1})$ and $D$, where we again bounded $ a^{1/(D+1)}e^{-c''A}$ by $e^{-c'''a}$.  Choosing $c$ to be the lesser of $c'$ and $c'''$ gives the result.
\end{proof}

Similarly we may bound the kernel $K_s(x,y)$ of $\op^{-s}$.
\begin{proposition}\label{prop:Kest} If $s<1$ then
\begin{equation*}
	|K_{s}(x,y)|\lesssim C(s) R(x,y)^{s(D+1)-D}
	\end{equation*}
\end{proposition}
\begin{proof}
One way to do this is to write the kernel $K_{s}(x,y)=\int_{0}^{\infty} \lambda^{-s} G_{\lambda}\, d\lambda$ and use the estimate from Proposition~\ref{prop:Rests}.  Apply~\eqref{eqn:integralest0A} with $A=\infty$, $\kappa=R(x,y)^{\gamma}$, $b=\gamma/(D+1)>0$ and $a=1-s-\frac{1}{D+1}$, where we only need to know $1-s>0$ because the rest of $a$ is from a factor $(1+\lambda)^{-1/(D+1)}$. 
\end{proof}

We will sometimes need $L^{p}(d\mu)$ estimates for kernels of this sort.
\begin{lemma}\label{lem:kerinLp}
If $K(x,y)$ is a kernel satisfying $|K(x,y)|\leq R(x,y)^{\eta}$ then $K(x,\cdot)\in L^{p}(X,\mu)$ for $\eta p>-D$.  If $B=B(x,R)$ is a ball of radius $R$ then 
\begin{equation*}
	\bigl\| K(x,\cdot) \bigr\|_{L^{p}(B,\mu)} \leq \frac{1}{p\eta+D} R^{\eta+D/p}.
	\end{equation*}
\end{lemma}
\begin{proof}
The measure is Alhfors regular with dimension $D$ in the resistance metric and the space has bounded diameter.  Accordingly the only obstacle to integrability is at $y=x$ and we may integrate radially with
\begin{equation*}
	\int_{B(x,R)} |K(x,y)|^{p}\, d\mu(y)
	\lesssim \int_{0}^{R} r^{p\eta+D} \frac{dr}{r}
	=\frac{1}{p\eta+D} R^{p\eta+D}.\qedhere
	\end{equation*}
\end{proof}

From the preceding two results the following is immediate.
\begin{corollary}\label{cor:bddness}
If $s(D+1)>\frac{D}{p}$ then $\Wsp\subset L^\infty$.
\end{corollary}

\section{Local behavior of functions in $\Wsp$}\label{Section:localbehavior}
We consider two quantities at a point $q\in V_{n}$.  Let $w$ be a word with $|w|=n$ and such that $q=F_{w}(q')$, $q'\in V_{0}$ and define
\begin{gather}
	\op_{m} u(q) = \sum_{x\sim_{m} q} \bigl( u(q)-u(x)\bigr)\\
	\delta_{m} u(F_{w}(q')) = \sum_{\substack{x\sim_{m} q\\x\in F_{w}(X)}} \bigl( u(q)-u(x)\bigr) \label{eqn:defnofdeltam}
	\end{gather}
where $\delta_m$ is only defined for $m\geq n$. Note that we write $\delta_{m}u(F_{w}(q'))$ rather than $\delta_{m}u(q)$ to emphasize the dependence on $w$.  Evidently $\op_{m}u(q)$ is obtained by summing $\delta_{m}u(F_{w}(q'))$ over the $n$-cells that meet at $q$, or more precisely those choices of $w$ of length $|w|=n$ and points $q'\in V_{0}$ such that $F_{w}(q')=q$.

Strichartz~\cite{S} used bounds of the type $|\op_{m}u|\lesssim r^{m\eta}$ to characterize H\"{o}lder-Zygmund spaces for a range of exponents $\eta$, and in particular to prove a Sobolev embedding theorem.  His Theorem~3.13 is a special case of the $Q=\infty$ statement in our next result.
\begin{theorem}\label{thm:holderembed}
Let $p\in(1,\infty]$, $s\in(0,1)$ and $q\in V_{n}$. If $s(D+1)>\frac{D}{p}$ and $u=\op^{-s}f$ for $f\in L^p$ then 
\begin{equation}\label{eqn:MainDeltamests}
	\|\op_{m}u(\cdot)\|_{l^{Q}(V_{m})}
	\lesssim
	\begin{cases}
	C(s) r^{m(s(D+1)-D/Q)} \|f\|_{p} &\text{if $1\leq Q\leq p$},\\
	C(s) r^{m(s(D+1)-D/p)} \|f\|_{p} &\text{if $p< Q\leq \infty$.}
	\end{cases}
	\end{equation}
\end{theorem}

The quantities $\delta_{m}u(F_{w}(q'))$ are related to the normal derivative $du(F_{w}(q'))$.  Using them we can give sufficient conditions for a function in $\Wsp$ to have a normal derivative at a vertex in $V_{\ast}$ and obtain it by integration against a kernel.
\begin{theorem}\label{thm:normalderiv}
Let $p,Q\in(1,\infty]$, $s\in(0,1)$ and $q=F_{w}(q')$, $q'\in V_{0}$ be a vertex in $V_{n}$. Define a normal derivative kernel by
\begin{equation*}
	dK_{s}(F_{w}(q'),y)=\int_{0}^{\infty} \lambda^{-s}dG_{\lambda}(F_{w}(q'),y)\,d\lambda.
	\end{equation*}
where $dG_\lambda$ is the normal derivative defined in~\eqref{eqn:d}.
If $s(D+1)>\frac{D}{p}+1$ then $dK_{s}(F_{w}(q'),\cdot)$  is in  $L^{p/(p-1)}(X,\mu)$ and for $m\geq n$
\begin{equation*}
	\biggl\| \delta_{m}u(F_{w}(\cdot)) - r^{m}\int dK_{s}(F_{w}(\cdot),y)f(y)\,d\mu(y) \biggr\|_{l^{Q}(V_{m})}
	\lesssim 
	\begin{cases}
	C(s) r^{m(s(D+1)-D/Q)} \|f\|_{p} &\text{if $1<Q\leq  p$},\\
	C(s) r^{m(s(D+1)-D/p)} \|f\|_{p} &\text{if $p< Q\leq\infty$.}
	\end{cases}
	\end{equation*}
In particular $\int dK_{s}(F_{w}(q'),y)f(y)\,d\mu(y)$ is the normal derivative of $u$ at $F_{w}(q')$ where $u=\op^{-s}f$ for $f\in L^{p}$,  in the sense that $r^{-m}$ times the $Q=\infty$ case of the estimate converges to zero as $m\to\infty$, see~\eqref{eqn:d}.
\end{theorem}

The proofs of Theorems~\ref{thm:holderembed} and~\ref{thm:normalderiv} occupy the remainder of this section.  We begin with an estimate of the  normal derivative of the resolvent kernel. 
\begin{proposition}\label{prop:dRest}
If $q=F_{w}(q')$ with $|w|=n$ and $q'\in V_{0}$ then for any $y\neq q$
\begin{equation}\label{eqn:dRest}
	\bigl|dG_{\lambda}(F_{w}(q'),y)\bigr|
	\lesssim  \bigl( 1+ r^{-n}(1+\lambda)^{-1/(D+1)} \bigr)\exp\bigl (-c R(q,y)^{\gamma}\lambda^{\gamma/(D+1)} \bigr).
	\end{equation}
\end{proposition}
\begin{proof}
Fix $\lambda>0$ and let $m$ be the integer part of $\frac{-\log\lambda}{(D+1)\log r}$ so $r^{-m(D+1)}= \mu^{-m(D+1)/D} \simeq \lambda$.  If $m<n$  let $\tilde{w}=w$ and otherwise let $\tilde{w}$ be the word of length $m$ such that $q\in F_{\tilde{w}}(V_{0})\subset F_{w}(X)$. Let $\psi_{m}$ be piecewise harmonic at scale $m$ with value $1$ at $q$ and zero on $V_{m}\setminus\{q\}$.  We will apply the local Gauss-Green formula to $G_{\lambda}$ and $\psi_{|\tilde{w}|}$ on the cell $F_{\tilde{w}}(X)$.

Recall that $\op G_{\lambda}(x,y)=-\lambda G_{\lambda}(x,y)$ away from $y$ and has a Dirac mass at $y$. Apply~\eqref{eqn:locGG} to see that if $y\not\in F_{\tilde{w}}(X)$ then
\begin{equation*}
	dG_{\lambda}(F_{w}(q'),y)
	= \int_{F_{\tilde{w}}(X)} -\psi_{|\tilde{w}|}(x) \lambda G_{\lambda}(x,y)\, d\mu(x) + \sum_{p'\in V_{0}} d\psi_{|\tilde{w}|}(p')(F_{\tilde{w}}(p')) G_{\lambda}(F_{\tilde{w}}(p'),y) 
	\end{equation*}
while if $y\in F_{\tilde{w}}(X)$ the expression needs only to be modified by adding $\psi_{|\tilde{w}|}(y)$ on the right side. Now $\|\psi_{|\tilde{w}|}\|_{\infty}\leq 1$ by the maximum principle and  each $d\psi_{|\tilde{w}|}(F_{w}(p'))\simeq r^{-|\tilde{w}|}$ by scaling. Since it is also the case that $\mu(F_{\tilde{w}}(X))=\mu^{|\tilde{w}|}$ we find
\begin{align*}
	\bigl|d G_{\lambda}(F_{w}(q'),y)\bigr|
	&\lesssim \lambda \mu^{|\tilde{w}|} \bigl\|   G_{\lambda}(\cdot,y)   \bigr\|_{L^{\infty}(F_{\tilde{w}}(X))}  +  r^{-|\tilde{w}|} \sum_{p'\in V_{0}} \bigl| G_{\lambda}(F_{\tilde{w}}(p'),y) \bigr|\\
	&\lesssim \bigl( \lambda \mu^{|\tilde{w}|} + r^{-|\tilde{w}|} \bigr)  \bigl\|   G_{\lambda}(\cdot,y)   \bigr\|_{L^{\infty}(F_{\tilde{w}}(X))} 
	\end{align*}
with the caveat that we must add $1$ to the right side if $y\in F_{\tilde{w}}(X)$.  Substituting the estimate of Proposition~\ref{prop:Rests} and using $\mu^{|\tilde{w}|}\leq \mu^{m}\simeq\lambda^{-D/(D+1)}$ we obtain
\begin{equation*}
	 \bigl|dG_{\lambda}(F_{w}(q'),y)\bigr|
	\lesssim \bigl( 1+ r^{-|\tilde{w}|}(1+\lambda)^{-1/(D+1)} \bigr) \exp \bigl( -c  \inf_{x\in F_{\tilde{w}}(X)} \lambda^{\gamma/(D+1)} R(x,y)^{\gamma} \bigr).
	\end{equation*}
This is valid even if $y\in F_{\tilde{w}}(X)$ because in this case the infimum in the exponent is zero, so the the Dirac mass term is absorbed into the estimate.  Now if $m\geq n$ then $r^{-\tilde{w}}=r^{-m}\simeq\lambda^{1/(D+1)}$ and the first factor is just a constant.  Otherwise $|\tilde{w}|=n>m$ and the $r^{-n}\lambda^{-1/(D+1)}$ term dominates.

To complete the proof we recall that $R(x,y)^{\gamma}$ is comparable to a metric and use the triangle inequality and the fact that   $R(q,x)\lesssim r^{|\tilde{w}|}\leq \lambda^{-1/(D+1)}$ if $x\in F_{\tilde{w}}(X)$ to obtain
\begin{equation*}
	\lambda^{\gamma/(D+1)} R(q,y)^{\gamma} 
	\leq \lambda^{\gamma/(D+1)} R(q,x)^{\gamma} + \lambda^{\gamma/(D+1)} R(x,y)^{\gamma} 
	\leq c' + \inf_{x\in F_{w}(X)} \lambda^{\gamma/(D+1)} R(x,y)^{\gamma}
	\end{equation*}
from which the result follows.
\end{proof}

\begin{corollary}\label{cor:dKdefined}
If $s\in(0,1)$ and $|w|=n$ the kernel
\begin{equation*}
	dK_{s}(F_{w}(q'),y)=\int_{0}^{\infty} \lambda^{-s}dG_{\lambda}(F_{w}(q'),y)\,d\lambda
	\end{equation*}
satisfies $\bigl| dK_{s}(F_{w}(q'),y)\bigr|\lesssim C(s) r^{-n}R(q,y)^{1-(D+1)(1-s)}+ C R(q,y)^{-(D+1)(1-s)}$.  If $p\in(1,\infty]$  and $s(D+1)>\frac{D}{p}+1$ this is in $L^{p/(p-1)}(d\mu(y))$.
\end{corollary}
\begin{proof}
The estimate from Proposition~\ref{prop:dRest} has two terms.  The term depending on $r$ is relevant for $\lambda<r^{-n(D+1)}$, so should be used in~\eqref{eqn:integralest0A} with $A=r^{-n(D+1)}$ and $a=1-s-\frac{1}{D+1}$, $b=\frac{\gamma}{D+1}>0$ and $\kappa=cR(q,y)^{\gamma}$. Note that we only need assume $1-s>0$ and not $a>0$ because part of our power of $\lambda$  is a factor $(1+\lambda)^{-1/(D+1)}$. Including the factor $r^{-n}$ from the integrand gives a result is bounded by $r^{-n}(cR(q,y))^{1-(1-s)(D+1)}$.  The other term (which only contains the exponential) can be done in the same way but with $a=1-s$ and $A=\infty$ to get a bound by $(cR(q,y))^{-(1-s)(D+1)}$.  Both pieces are then in $L^{p/(p-1)}(d\mu(y))$ provided $(1-s)(D+1)p/(p-1)<D$ by Lemma~\ref{lem:kerinLp}.
\end{proof}

We now examine the difference operators $\op_{m}$ and $\delta_{m}$.  For $u=\op^{-s}f=\int K_{s}(x,y)f(y)\,d\mu(y)$ write
\begin{equation}\label{eqn:DmGusingLap}
	\op_{m}u(q) = \int_{X} \op_{m}K_{s}(q,y)f(y)\,d\mu = \int_{X}\int_{0}^{\infty} \lambda^{-s} \op_{m}G_{\lambda}(q,y)f(y)\,d\lambda\,d\mu(y)
	\end{equation}
and similarly,  if $s$ and $p$ are as in Corollary~\ref{cor:dKdefined} so that $du(F_{w}(q'))=\int dK_{s}(q,y)f(y)\,d\mu(y)$ is well defined,
\begin{align}
	\lefteqn{\delta_{m}u(F_{w}(q')) -r^{m}du(F_{w}(q'))}\quad& \notag\\
	&=\int_{X} \bigl(\delta_{m}K_{s}(F_{w}(q'),y) - r^{m}dK_{s}(F_{w}(q'),y) \bigr) f(y)\,d\mu  \notag\\
	&= \int_{X}\int_{0}^{\infty} \lambda^{-s} \bigl( \delta_{m}G_{\lambda}(F_{w}(q'),y)-r^{m}dG_{\lambda}(F_{w}(q'),y)\bigr)f(y)\,d\lambda\,d\mu(y).  \label{eqn:dmGusingLap}
	\end{align}
We wrote these expressions in this form because both quantities are  readily estimated.  First we note trivial estimates that do not account for any cancellation.  By summing~\eqref{eqn:Rests} over the $m$-scale neighbors of $q$
\begin{equation}\label{eqn:trivialDmG}
	|\op_{m}G_{\lambda}(q,y) | \lesssim (1+\lambda)^{-1/(D+1)} \exp\Bigl( - c \min_{x\sim_{m}q} R(x,y)^{\gamma} \lambda^{\gamma/(D+1)} \Bigr).
	\end{equation}
Similarly the crude bound from~\eqref{eqn:Rests} and~\eqref{eqn:dRest} gives (using $m>n$)
\begin{equation}\label{eqn:trivialdmG}
	\bigl|\delta_{m}G_{\lambda}(F_{w}(q'),y) - r^{m} dG_{\lambda}(F_{w}(q'),y) \bigr|
	\lesssim \bigl( r^{m} + (1+\lambda)^{-1/(D+1)} \bigr) \exp\Bigl( - c \min_{x\sim_{m}q} R(x,y)^{\gamma} \lambda^{\gamma/(D+1)} \Bigr).
	\end{equation}

These estimates cannot be substantially improved if $y$ is close to $q$, but if it is not then we can estimate using regularity of $G_\lambda$.
\begin{proposition}\label{prop:interpests}
Fix $q\in V_{n}$, $m>n$, $\theta\in[0,1]$.  If $y$ is not in any $(m-1)$-cell containing $q$ then
\begin{gather*}
	\bigl|\op_{m}G_{\lambda}(q,y)\bigr|\lesssim \frac{(r\mu)^{m\theta} \lambda^{\theta}}{(1+\lambda)^{1/(D+1)}} \exp \Bigl( -c R(q,y)^{\gamma} 		\lambda^{\gamma/(D+1)} \Bigr) \\
	\bigl|\delta_{m}G_{\lambda}(F_{w}(q'),y) - r^{m} dG_{\lambda}(F_{w}(q'),y) \bigr|
	\lesssim \Bigl( \frac{(r\mu)^{m\theta}\lambda^{\theta}}{(1+\lambda)^{1/(D+1)}} +  r^{m}\mu^{m\theta}\lambda^{\theta D/(D+1)}\Bigr) \exp \Bigl( -c R(q,y)^{\gamma} \lambda^{\gamma/(D+1)} \Bigr)
\end{gather*}
\end{proposition}
\begin{proof}
Let $\psi_{m}$ be piecewise harmonic of scale $m$ with value $1$ at $q$ and zero at all other points of $V_{m}$. Then $\op_{m}u=\DF(u,\psi_{m})=\int (\op u)\psi_{m}$ if $u$ is sufficiently regular.  Now with $u=G_{\lambda}$ we have $\op G_{\lambda} = -\lambda G_{\lambda}$ on the support of $\psi_{m}$ by our assumption on $y$.  Thus
\begin{equation*}
	r^{-m}\bigl|\op_{m}G_{\lambda}(q,y)\bigr|
	=\Bigl|\int (-\lambda G_{\lambda}(x,y)) \psi_{m}(x)\, d\mu(x)\Bigr|
	\leq \lambda \mu^{m} (1+\lambda)^{-1/(D+1)} \exp \Bigl( -c R(q,y)^{\gamma} \lambda^{\gamma/(D+1)} \Bigr)
	\end{equation*}
because $|\psi_{m}|\leq 1$ by the maximum principle and $|G_{\lambda}|$ may be estimated using~\eqref{eqn:Rests}.  Note that in the estimate of $|G_{\lambda}|$ we must take the supremum over the support of $\psi_{m}$, but $R(x,y)^{\gamma}\geq cR(q,y)^{\gamma}$ on this set by the triangle inequality and our hypothesis that $y$ is separated from the $m$-cell containing $q$.  The desired estimate comes from the product of the $\theta$ power of this inequality with the $(1-\theta)$ power of~\eqref{eqn:trivialDmG}.

The proof for $\delta_{m}G_{\lambda}$ is almost identical.  Since $q=F_{w}(q')\in V_{n}$, for $m>n$ there is a unique $m$-cell $F_{\tilde{w}}(X)$ contained in $F_{w}(X)$.  Following the same reasoning as for $\op_{m}$ but restricting to $F_{\tilde{w}}(X)$ we find from the local Gauss-Green formula~\eqref{eqn:locGG} that
\begin{equation*}
	r^{-m} \delta_{m}G_{\lambda}(F_{w}(q'),y) - dG_{\lambda}(F_{w}(q'),y)
	= \int_{F_{\tilde{w}}(X)} (-\lambda G_{\lambda}(x,y)) \psi_{m}(x)\, d\mu(x)
	\end{equation*}
from which point we make the same estimate as before, take the $\theta$ power and multiply by the $(1-\theta)$ power of~\eqref{eqn:trivialdmG} to complete the proof.
\end{proof}

\begin{corollary}\label{cor:DdmKsests}
Fix $q\in V_{n}$, $m>n$, $s\in(0,1)$ and $\theta\in[0,1]$.  If $y$ is not in any $(m-1)$-cell containing $q$ then
\begin{align*}
	\bigl|\op_{m}K_{s}(q,y) \bigr|
	&\lesssim  C(s) (r\mu)^{m\theta} R(q,y)^{(s-\theta)(D+1)-D},\\
	\bigl| \delta_{m}K_{s}(F_{w}(q'),y) - r^{m} dK_{s}(F_{w}(q'),y) \bigr|
	&\lesssim C(s)(r\mu)^{m\theta} R(q,y)^{(s-\theta)(D+1)-D}. 
	\end{align*}
\end{corollary}
\begin{proof}
Integrate the estimates from Proposition~\ref{prop:interpests} against $\lambda^{-s}$ to obtain $\op_{m}K_{s}$ as in~\eqref{eqn:DmGusingLap} and $(\delta_{m}-r^{m}dK_{s})$ as in~\eqref{eqn:dmGusingLap}.  In the first case we may use~\eqref{eqn:integralest0A} with $A=\infty$, $\kappa=R(q,y)^{\gamma}$, $b=\gamma/(D+1)$ and $a=1-s+\theta- \frac{1}{D+1}>0$. Then the integral is bounded by $R(q,y)^{1-(D+1)(1-s+\theta)}$, which combined with the factor $(r\mu)^{m\theta}$ gives the stated bound for $\op_{m}K_{s}$.

In the $\delta_{m}$ case we split the integral over $\lambda\in[0,A]$ and in $[A,\infty)$ with $A=r^{-m(D+1)}$.  Observe that on $[0,A]$ the first term from Proposition~\ref{prop:interpests} dominates and we can use~\eqref{eqn:integralest0A} with $a$,$b$,$\kappa$ as before to obtain the same bound $(r\mu)^{m\theta}R(q,y)^{1-(D+1)(1-s+\theta)}$.  On $[A,\infty)$ the second term dominates and  we use~\eqref{eqn:integralestAinfty} with the same $b$ and $\kappa$ but $a=1-s+\frac{\theta D}{D+1}$, so the integral is bounded by $r^{-m(D+1)(1-s)-mD\theta}\exp\bigl(-c( r^{-m}R(q,y))^{\gamma}\bigr)$. Putting in the powers $r^{m}\mu^{m\theta}=r^{m(1+D\theta)}$ the resulting bound may be written  $(r\mu)^{m\theta}r^{m((s-\theta)(D+1)-D)}\exp\bigl(-c( r^{-m}R(q,y))^{\gamma}\bigr)$.  However $y$ is not in any $(m-1)$-cell containing $q$, so $R(q,y)\geq C'r^{m}$, which makes the exponential a constant and implies  $(r\mu)^{m\theta}r^{m((s-\theta)(D+1)-D)}$ is smaller than $(r\mu)^{m\theta}R(q,y)^{(s-\theta)(D+1)-D}$.
\end{proof}

\begin{proof}[\protect{Proof of Theorem~\ref{thm:holderembed}}]
Fix $p\in(1,\infty]$, $s\in(0,1)$ and $q\in V_{n}$ and suppose $s(D+1)>\frac{D}{p}$. Let $X_{0}$ be the union of the $(m-1)$-cells containing $q$. This contains a disc of radius $cr^m$ around $q$, and for $j\geq 1$ we let $X_{j}=\{x:c2^{j}r^{m}< R(x,y)\leq c2^{j+1}r^{m}\}\setminus X_{0}$ be the part of the  annulus centered at $y$ that is not in $X_{0}$.  Evidently $X=\cup_{j\geq0}X_{j}$.  Break the integration $\op_{m}u(q)=\int \op_{m}K_{s}(q,y)f(y)\, d\mu(y)$ according to the $X_{j}$ and use Minkowski's and H\"{o}lder's inequalities to obtain
\begin{align}
	\bigl\|\op_{m}u (q)\bigr\|_{l^{Q}(V_{m})}
	&= \biggl\| \sum_{j} \int_{X_{j}} f(y) \op_{m}K_{s} (q,y) \, d\mu(y) \biggr\|_{l^{Q}(V_{m})}\notag \\
	&\leq \sum_{j}  \Bigl\| \int_{X_{j}}  f(y) \op_{m}K_{s} (q,y) \, d\mu(y) \Bigr\|_{l^{Q}(V_{m})} \notag\\
	&\leq \sum_{j} \biggl\| \Bigl( \int_{X_{j}} |f(y)|^{p} \, d\mu \bigr)^{1/p} \Bigl( \int_{X_{j}} \bigl| \op_{m}K_{s}(q,y)\bigr|^{p/(p-1)} \, d\mu \Bigr)^{(p-1)/p} \biggr\|_{l^{Q}(V_{m})}. \label{eqn:minkowbdonDeltamu}
	\end{align}

Now on $X_{0}$  we can only bound $\op_{m}G_{\lambda}$ as in~\eqref{eqn:trivialDmG}. Since this is the same bound as for $G_{\lambda}$ the corresponding bound on $\op_{m}K_{s}$ is the same as for $K_{s}$, which by  Proposition~\ref{prop:Kest} is $C(s)R(q,y)^{s(D+1)-D}$.  Applying Lemma~\ref{lem:kerinLp} this power of $R(q,y)$ is in $L^{p/(p-1)}(d\mu(y))$ and
\begin{equation*}
	\biggl( \int_{X_{0}} \bigl| \op_{m}K_{s}(q,y)\bigr|^{p/(p-1)} \, d\mu \biggr)^{(p-1)/p}
	\lesssim C(s)r^{m(s(D+1)-D +\frac{D(p-1)}{p})} 
	= C(s) r^{m( s(D+1)-\frac{D}{p})}.
	\end{equation*}
On $X_{j}$ we are outside the $(m-1)$-cells containing $q$ so  so we use Corollary~\ref{cor:DdmKsests} with $\theta=1$ to see
\begin{equation*}
	\bigl|\op_{m}K_{s}(q,y) \bigr|
	\lesssim C(s) (r\mu)^{m} R(q,y)^{(s-1)(D+1)-D}
	= C(s)  r^{m(D+1)} R(q,y)^{(s-1)(D+1)-D}
	\end{equation*}
Then on $X_{j}$ we have control on $R(q,y)$ and from Ahlfors regularity the measure is at most a multiple of $\bigl( 2^{j}r^{m}\bigr)^{D}$, so that
\begin{equation*}
	\int_{X_{j}} \bigl| \op_{m}K_{s}(q,y)\bigr|^{p/(p-1)} \, d\mu 
	\lesssim C(s) r^{m(D+1)p/(p-1)} \Bigl( 2^{j}r^{m}\Bigr)^{((s-1)(D+1)-D)p/(p-1)} \Bigl( 2^{j+1}r^{m} \Bigr)^{D}.
	\end{equation*}
We summarize these bounds on the $X_{j}$ integrals as
\begin{equation*}
	\biggl( \int_{X_{j}} \bigl| \op_{m}K_{s}(q,y)\bigr|^{p/(p-1)} \, d\mu \biggr)^{(p-1)/p}
	\lesssim r^{m(s(D+1)-D/p)} 2^{j((s-1)(D+1)-D)}
	\end{equation*}
and combine them with~\eqref{eqn:minkowbdonDeltamu} to obtain
\begin{equation}\label{eqn:Deltamu}
	\bigl\|\op_{m}u (q)\bigr\|_{l^{Q}(V_{m})}
	\lesssim r^{m(s(D+1)-D/p)} \sum_{j\geq0} 2^{j((s-1)(D+1)-D/p)} \biggl\| \Bigl( \int_{X_{j}} |f(y)|^{p} \, d\mu \Bigr)^{1/p} \biggr\|_{l^{Q}(V_{m})}.
	\end{equation}

The dependence of $\int_{X_{j}}|f|^{p}$ on $q\in V_{m}$ is through $X_{j}=X_{j}(q)$.  If $Q\leq p$ then H\"{o}lders inequality gives the bound
\begin{equation*}
	\biggl\| \Bigl( \int_{X_{j}} |f(y)|^{p} \, d\mu \Bigr)^{1/p} \biggr\|_{l^{Q}(V_{m})}
	\leq r^{-mD(\frac{1}{Q}-\frac{1}{p})}  \biggl( \sum_{q\in V_{m}} \int_{X} |f(y)|^{p} \mathds{1}_{X_{j}(q)}(y) \, d\mu(y)   \biggr)^{1/p}
	\end{equation*}
because the number of points in $V_{m}$ bounded by the number of $m$-cells, which is at most a multiple of $\mu^{-m}=r^{-mD}$.  To proceed we notice that for fixed $y$ the $q$ such that $y\in X_{j}(q)$ have $R(q,y)\leq c2^{j+1}r^{m}$.  The number of such $q$ is bounded by a multiple of the number of $m$-cells in the corresponding ball around $y$.  Since these cells are disjoint and of measure $\mu^{m}=r^{mD}$ and the ball has measure bounded by a multiple of $\bigl(2^{j}r^{m}\bigr)^{D}$ by Ahlfors regularity, the number of $q$ so $y\in X_{j}(q)$ is bounded by $2^{jD}$.  Thus
\begin{equation*}
	\biggl( \sum_{q\in V_{m}} \int_{X} |f(y)|^{p} \mathds{1}_{X_{j}(q)}(y) \, d\mu(y)   \biggr)^{1/p}
	\lesssim 2^{jD/p} \|f\|_{p}.
	\end{equation*}
Combining this bound for $Q\leq p$ with the fact that the $l^{Q}$ norm is dominated by the $l^{p}$ norm when $Q>p$ we have
\begin{equation*}
	\biggl\| \Bigl( \int_{X_{j}} |f(y)|^{p} \, d\mu \bigr)^{1/p} \biggr\|_{l^{Q}(V_{m})}
	\leq\min \bigl\{1,  r^{-mD(\frac{1}{Q}-\frac{1}{p})}\bigr\} 2^{jD/p} \|f\|_{p}.
	\end{equation*}
We can substitute this into~\eqref{eqn:Deltamu} to see
\begin{align*}
	\bigl\|\op_{m}u (q)\bigr\|_{l^{Q}(V_{m})}
	&\lesssim \|f\|_{p}  r^{m(s(D+1)-D/p)}\min \bigl\{1,  r^{-mD(\frac{1}{Q}-\frac{1}{p})} \bigr\} \sum_{j\geq0} 2^{j((s-1)(D+1)-D/p)}  2^{jD/p}\\
	&= \begin{cases}
	\|f\|_{p}   r^{m(s(D+1)-D/p)} &\text{ if $p<Q\leq\infty$,}\\
	\|f\|_{p}   r^{m(s(D+1)-D/Q)} &\text{ if $1\leq Q\leq p$}
	\end{cases}
	\end{align*}
which is~\eqref{eqn:MainDeltamests}.
\end{proof}

\begin{proof}[\protect{Theorem~\ref{thm:normalderiv}}]
The stated properties of the integral giving $dK_{s}$ were proved in Corollary~\ref{cor:dKdefined} and the same argument as in the previous proof yields the result.
\end{proof}

\section{Failure of the algebra property}\label{Section:algebraprop}
In~\cite{BST} Ben-Bassat,  Strichartz and Teplyaev proved that the square of a function in $\op^{-1}L^{\infty}$ which has non-zero normal derivative at a point of $V_{m}$ is not in $\op^{-1}L^{\infty}$.  The heart of their argument is the fact that if $f$ has non-zero normal derivative at $q\in V_{\ast}$ then $f$ is comparable to a linear function in the resistance metric near $q$, but the property of being in $\op^{-1}L^{\infty}$ implies the difference $\op_{m}$ is smaller than the square of the resistance.  It is apparent from the results of~\cite{BST} that this argument can be generalized to some other \Wsp, though their methods only work for $s=1$.   The following generalizes their main argument to \Wsp\ for a much larger collection of $s$ and $p$.
\begin{theorem}\label{thm:nosquares}
Let $p\in(1,\infty]$ and $s\in(0,1)$ such that $s(D+1)>\frac{D}{p}+2$.  Suppose $u\in\Wsp$ and there is $q=F_{w}(q')\in V_{\ast}$, $q'\in V_{0}$ at which $du(F_{w}(q'))\neq0$.  Then $u^{2}\not\in\Wsp$.
\end{theorem}
\begin{proof}
Let $v(x)=u(x)-u(q)$. Write $\op_{m}u^{2}(q)=\op_{m} v^{2}(q) + 2u(q)\op_{m}v(q)$, and observe that $\op_{m}v(q)=\op_{m}u(q)$ because the functions differ only by a constant.  Since $u\in\Wsp$, Theorem~\ref{thm:holderembed} implies there is $C$ so $|\op_{m}u(q)|\leq C r^{m(s(D+1)-D/p)}$.  If $u^{2}\in\Wsp$ then the same estimate would hold for $|\op_{m}u^{2}(q)|$ and therefore for $|\op_{m}v^{2}(q)|$.  However $v(q)=0$, so 
\begin{equation*}
	\bigl(\delta_{m}u(q)\bigr)^{2}
	\leq L \sum_{\substack{x\sim_{m}q\\x\in F_{w}(X)}} (u(q)-u(x))^{2}
	\leq L \sum_{x\sim_{m}q} v(x)^{2}
	= \sum_{x\sim_{m}q}(v^{2}(x)-v^{2}(q))
	= \bigl|\op_{m}v^{2}(q)\bigr|
	\end{equation*}
and we conclude $|\delta_{m}u(q)|^{2}\leq C r^{m(s(D+1)-D/p)}$.  Since $s(D+1)-\frac{D}{p}>2$ we may compare with the $Q=\infty$ case of Theorem~\ref{thm:normalderiv} to find $du(F_{w}(q'))=0$
\end{proof}

The proof of the preceding theorem generalizes easily to show that composing an element of \Wsp\ that has a non-vanishing normal derivative with a function that is convex (or concave) with a H\"{o}lder estimate on the convexity produces functions that cannot be in \Wsp.
\begin{corollary}\label{cor:noconvexfns}
Let $p\in(1,\infty]$, $s\in(0,1)$ and $s(D+1)-\frac{D}{p}>1$.  Suppose $u\in\Wsp$ and there is $q=F_{w}(q')\in V_{\ast}$, $q'\in V_{0}$ at which $du(F_{w}(q'))\neq0$. If $\Phi$ is a function with bounded derivative and which satisfies the following convexity condition at $u(q)$:  there is $1\leq\xi<s(D+1)-\frac{D}{p}$ and $C>0$ such that 
\begin{equation*}
	\Phi(y)-\Phi(u(q)) -\Phi'(u(q))(y-u(q)) \geq C |y-u(q)|^{\xi}.
	\end{equation*}
then $\Phi\circ u\not\in\Wsp$.
\end{corollary}
\begin{proof}
By H\"{o}lder's inequality and the assumption on $\Phi$
\begin{align*}
	|\delta_{m}u(F_{w}(q'))|^{\xi}
	&\leq L^{\xi-1} \sum_{\substack{x\sim_{m}q\\x\in F_{w}(X)}} |u(x)-u(q)|^{\xi} \leq \sum_{x\sim_{m}q}  |u(x)-u(q)|^{\xi}\\
	&\leq \frac{1}{C} \sum_{x\sim_{m}q}\bigl(  \Phi(u(x))-\Phi(u(q)) -\Phi'(u(q))(u(x)-u(q)) \bigr)\\
	&\leq \frac{1}{C} \bigl| \op_{m}(\Phi\circ u)(q)\bigr| + \frac{\Phi'(u(q))}{C}\bigl| \op_{m}u(q)\bigr|.
	\end{align*}
If $\Phi\circ u\in\Wsp$ then both terms on the right are bounded by a multiple of $mr^{m(S(D+1)-\frac{D}{p})}$. From our assumption on $\xi$ we conclude that $\delta_{m}u(F_{w}(q'))=o(r^{m})$ and thus $du(F_{w}(q'))=0$.
\end{proof}

The preceding results are only interesting when we know something more about functions whose normal derivatives vanish on $V_{m}$.  Fortunately we can obtain this from the $Q=2$ case of Theorem~\ref{thm:normalderiv} using the following result.
\begin{proposition}\label{prop:DFasdeltam}
If $u$ is a function on $X$ for which $\bigl\| \delta_{m}u(x) \bigr\|_{l^{2}(V_{m})}=o(r^{m/2})$ then $u$ is constant.  If, in addition, $u\in\Wsp$ then $u\equiv0$. 
\end{proposition}
\begin{proof}
Recall that the Dirichlet form was obtained as
\begin{equation*}
	\DF(u)
	=\lim_{m\to\infty}r^{-m}\sum_{x\sim_{m}y}(u(x)-u(y))^{2}
	=\lim_{m\to\infty}r^{-m}\sum_{\{w:|w|=m\}} \sum_{x,y\in F_{w}(V_{0})} (u(x)-u(y))^{2}
	\end{equation*}
where we have re-written the sum is over all edges of the $m$-scale graph as a sum over cells using that $x\sim_{m}y \iff x,y\in F_{w}(V_{0})$ for some $w$ with $|w|=m$. Now at $x\in F_{w}(V_{0})$ we have from~\eqref{eqn:defnofdeltam}
\begin{equation*}
	\delta_{m} u(x) = \sum_{z\in F_{w}(V_{0})} (u(x)-u(z)) = |V_{0}| u(x) - \sum_{z\in F_{w}(V_{0})} u(z)
	\end{equation*}
where $|V_{0}|$ denotes the number of points in $V_{0}$.  Hence $\delta_{m}u(x)-\delta_{m}u(y)=|V_{0}|(u(x)-u(y))$, and therefore
\begin{align*}
	\DF(u)
	&=|V_{0}|^{-1} \lim_{m\to\infty}r^{-m}\sum_{\{w:|w|=m\}} \sum_{x,y\in F_{w}(V_{0})} (\delta_{m}u(x)-\delta_{m}u(y))^{2}\\
	&\leq2 |V_{0}|^{-1} \lim_{m\to\infty}r^{-m}\sum_{\{w:|w|=m\}} \sum_{x,y\in F_{w}(V_{0})} \bigl( |\delta_{m}u(x)|^{2}+ |\delta_{m}u(y)|^{2}\bigr)\\
	&= 4 |V_{0}|^{-1} \lim_{m\to\infty}r^{-m} \bigl\| \delta_{m}u(x) \bigr\|_{l^{2}(V_{m})}^{2}.
	\end{align*}
From our hypotheses we now find $\DF(u)=0$, whereupon $u$ is constant.  If also $u\in\Wsp$ then $u=0$ on $V_{0}$, so $u\equiv0$.
\end{proof}

\begin{corollary}
Suppose $p\in(1,\infty]$ and $s\in(1/2,1)$ with $s(D+1)-\frac{D}{p}>1$.  If $u\in\Wsp$ has $du(F_{w}(q))=0$ for all finite words $w$ and all $q\in V_{0}$ then $u\equiv0$.
\end{corollary}
\begin{proof}
The assumption $s(D+1)-\frac{D}{p}>1$ is made to ensure the normal derivative $dK_s$ from Theorem~\ref{thm:normalderiv} is integrable against $f$.  Using the $Q=2$ estimate from that result we see that if $p\in(1,2]$ then $\bigl\| \delta_{m}u(x) \bigr\|_{l^{2}(V_{m})}=O(r^{m(s(D+1)-D/p)})=o(r^{m})$, so from  Proposition~\ref{prop:DFasdeltam} we get $u\equiv0$.  The corresponding estimate when $p\in(2,\infty)$ is that $\bigl\| \delta_{m}u(x) \bigr\|_{l^{2}(V_{m})}=O(r^{m(s(D+1)-D/2)})$, and for $p=\infty$ is the same but with an extra factor of $m$.  In either case we can apply Proposition~\ref{prop:DFasdeltam} to get $u\equiv0$ because $s(D+1)-\frac{D}{2}>\frac{1}{2}$ is simply $s>\frac{1}{2}$.
\end{proof}

\begin{corollary}
If  $p\in(1,\infty]$ and $s\in(1/2,1)$ with $s(D+1)-\frac{D}{p}>2$ then \Wsp\ is not an algebra.
\end{corollary}
\begin{proof}
Applying Theorem~\ref{thm:nosquares} we find that any function in \Wsp\ with square in \Wsp\ has vanishing normal derivative on $V_{m}$ for all $m$, so by the previous corollary it is identically zero.  However \Wsp\ contains many non-zero functions.  For example, by results of~\cite{RST}, for any compact $K\subset X$ and open neighborhood $U\supset K$ there is  a smooth $u$ which is $1$ on $K$, $0$ outside $U$. In particular this $u$ has continuous $\op u$ so is in \Wsp.
\end{proof}
Similarly, but using Corollary~\ref{cor:noconvexfns} instead of Theorem~\ref{thm:nosquares} we have
\begin{corollary}
If  $p\in(1,\infty]$ and $s\in(1/2,1)$ with $s(D+1)-\frac{D}{p}>\xi\geq1$ then \Wsp\ is not closed under the action of $\Phi$ as in Corollary~\ref{cor:noconvexfns}.
\end{corollary}

\section{Specific fractal examples}\label{Section:Examples}

Our arguments are applicable to the classical Sierpinski Gasket, \SG, which is the unique non-empty compact fixed set of the iterated function system $\{F_{j}=\frac{1}{2}(x+p_{j})\}_{j=0,1,2}$ where the points $p_{j}$ are vertices of an equilateral triangle in $\mathbb{R}^{2}$.  This fractal is very well-studied (see for example~\cite{bobsbook}) and has $r=\frac{3}{5}$ and $\mu=\frac{1}{3}$.  The upper heat kernel estimates (originally from~\cite{BP}) and resolvent kernel estimates (for $\lambda>0$) are as in Section~\ref{sec:Fractalsandresolvent} with $\gamma=\frac{\log 2}{\log(5/3)}$ and $D=\frac{\log3}{\log(5/3)}$.  Note that then $R(x,y)^{\gamma}$ is comparable to the Euclidean path metric on the fractal.  The case $s=1$, $p=\infty$ of the following theorem was proved in~\cite{BST}.

\begin{theorem}
On the Sierpinski Gasket, \Wsp\ is not an algebra if $p\in(1,\infty]$ and $s\in(1/2,1)$ with $s\log 5-\frac{1}{p}\log3>2\log (5/3)$.
\end{theorem}
\begin{remark}
Using  Corollary~\ref{cor:bddness} we see that in the language of Theorem~\ref{thm:mainresult} neither $W^{\alpha,p}$ nor $\dot{W}^{\alpha,p}\cap L^\infty$ are algebras on \SG\ for when
\begin{equation*}
	\max\biggl\{1, \frac{4\log(5/3)}{\log5} + \frac{2\log3}{\log5} \frac{1}{p} \biggr\} <\alpha<2.
	\end{equation*}
 Note that this interval is non-empty if $p>\frac{\log3}{2\log3-\log5}$.
\end{remark}
Our approach also works on a generalization of the Vicsek set.  Following the notation of Barlow in~\cite{B} we work in $\mathbb{R}^{N}$, $N\geq2$ and let $L\geq1$ be an integer.  Let $X_{0}=[0,1]^{N}$ be the unit cube, $V_{0}=\{q_{i}\}_{i=1}^{2^N}$ be its vertices and $x_{0}=(\frac{1}{2},\dotsc,\frac{1}{2})$ its center.  By dividing each axial direction into $L+1$ equal pieces subdivide $X_{0}$ into cubes and let $X_{1}$ be the union of the $2^{N}L+1$ cubes with centers on the lines from $x_{0}$ to each of the $q_{i}$.  Let $\{F_{j}\}_{j=1}^{2^{N}L+1}$ be the orientation preserving linear maps from $X_{0}$ to each cube in $X_{1}$ and let \VLN\  be the fixed set of the resulting iterated function system.  Evidently the self-similar measure has $\mu=(2^{N}L+1)^{-1}$.  It is easy to prove that the construction of a self-similar resistance form from Section~\ref{sec:Fractalsandresolvent} works with $r=(2L+1)^{-1}$.  One way to do so is to consider a function on $V_{0}$ with value $a_{i}$ at $q_{i}$ and $\sum_{i}a_{i}=0$, and suppose it extends so the value at the corresponding point of the central cube of $X_{1}$ is $b_{i}$. One verifies that each string of $L$ cubes in $X_{1}$ from the central cube to $q_{i}$ contributes $2^{N-1} L^{-1}(a_{i}-b_{i})^{2}$ to the $\DF_{1}$ form while the central cube contributes $\sum_{j<k}(a_{i}-b_{i})^{2}$.  Minimizing over the $b_{i}$ gives $\sum_{k\neq i}(b_{i}-b_{k})=2^{N-1}L^{-1}(a_{i}-b_{i})$ for each $i$; this has unique solution $b_{i}=(2L+1)^{-1}a_{i}$, which gives $\DF_{1}=(2L+1)^{-1}\DF_{0}$.  Note that this implies the resistance metric is comparable to the Euclidean metric.  Since the minimal extension of a constant function is constant we have also obtained a description of all harmonic functions.

The upper heat kernel estimates on \VLN\ depend on $L$ and $N$.  In the simplest case ($N=2$, $L=1$) they were proved in~\cite{Krebs}, while the version we need follows by applying standard results (such as those in~\cite{GrigoryanTelcs}) to some estimates proved in~\cite{B}. On \VLN\ they have the form provided in Section~\ref{sec:Fractalsandresolvent} with $\gamma=1$ and  $D=\frac{\log(2^{N}L+1)}{\log(2L+1)}$.

A significant feature of this class of examples is that by sending $N\to\infty$ we have $D\to\infty$, thus there is a $D$ for which the condition $s(D+1)-\frac{D}{p}>2$ is satisfied as soon as $sp>1$.  Our statement about when \Wsp\ is an algebra is as follows.
\begin{theorem}\label{thm:VNLexamples}
On \VLN, \Wsp\ is not an algebra if $p\in(1,\infty]$ and $s\in(1/2,1)$ with $s\log((2^{N}L+1)(2L+1))-\frac{1}{p}\log(2^{N}L+1) >2\log(2L+1)$.
\end{theorem}

In particular, if $p<\infty$ and $sp>1$ or if $p=\infty$ and $s>1/2$ we can take $N$ so large that the last condition holds, proving the next result.
\begin{theorem}\label{thm:VNLexamplesDlimit}
If  $p\in(1,\infty]$ and $s\in(1/2,1)$ with $sp>1$ there is $N$ such that \Wsp\ on \VLN\ is not an algebra. 
\end{theorem}

\section{Acknowledgements}
This research was undertaken while T.C. was employed by the Mathematical Sciences Institute at the Australian National University.  The second author thanks the MSI for its hospitality during the period in which this paper was written.


\begin{thebibliography}{99}

\bibitem{BBR}
   \newblock Nadine Badr, Fr\'{e}d\'{e}ric Bernicot and Emmanuel Russ,
   \newblock Algebra properties for Sobolev spaces -- applications to semilinear PDEs on manifolds,
   \newblock \emph{J. Anal. Math.} \textbf{118} (2012), no. 2, 509--544.

\bibitem{B}
     \newblock   Martin T. Barlow,
     \newblock   Which values of the volume growth and escape time exponent are possible for a graph?,
    \newblock  \emph{Rev. Mat. Iberoamericana}, \textbf{20} (2004), 1--31.

\bibitem{BP}
     \newblock   Martin T. Barlow  and Edwin A. Perkins, 
     \newblock    Brownian motion on the Sierpi\'nski gasket,
    \newblock  \emph{Probab. Theory Related Fields}, \textbf{79} (1988), 543--623.

\bibitem{BST} 
   \newblock    Oren Ben-Bassat, Robert S. Strichartz  and Alexander Teplyaev,
   \newblock     What is not in the domain of the Laplacian on Sierpinski gasket type fractals,
  \newblock      \emph{J. Funct. Anal.,} \textbf{166} (1999), 197--217.

\bibitem{BCF}
   \newblock Fr\'{e}d\'{e}ric Bernicot, Thierry Coulhon and Dorothee Frey,
   \newblock Sobolev algebras through heat kernel estimates,
   \newblock Preprint, 2015.

\bibitem{CS}
   \newblock F. Cipriani and J.-L. Sauvageot,
   \newblock Derivations as square roots of Dirichlet forms,
   \newblock  \emph{J. Funct. Anal.}  \textbf{201} (2003), 78--120.

\bibitem{CRT}
   \newblock Thierry Coulhon, Emmanuel Russ and Val\'{e}rie Tardivel-Nachef,
   \newblock  Sobolev algebras on Lie groups and Riemannian manifolds,
   \newblock  \emph{Amer. J. Math.} \textbf{123} (2001), no. 2, 283--342. 

\bibitem{FHK}
     \newblock  Pat J. Fitzsimmons, Ben M. Hambly  and Takashi  Kumagai, 
     \newblock   Transition density estimates for {B}rownian motion on affine
              nested fractals,
    \newblock    \emph{Comm. Math. Phys.}, \textbf{165} (1994), 595--620.

\bibitem{FOT} 
     \newblock  Masatoshi Fukushima  and Tadashi Shima, 
    \newblock   On a spectral analysis for the {S}ierpi\'nski gasket,
  \newblock   \emph{Potential Anal.}, \textbf{1} (1992),   1--35.


\bibitem{GrigoryanTelcs} 
   \newblock     Alexander Grigor'yan and Andras Telcs, 
   \newblock     Two-sided estimates of heat kernels on metric measure spaces,
   \newblock     \emph{Ann. Probab.} \textbf{40} (2012), 1212--1284.

\bibitem{GK}
   \newblock Archil Gulisashvili and Mark Kon,
   \newblock Exact smoothing properties of Schr\"{o}dinger semigroups,
   \newblock  \emph{Amer. J. Math.} \textbf{118} (1996), no. 6, 1215--1248. 

\bibitem{H}
   \newblock John~E. Hutchinson,
   \newblock Fractals and self-similarity,
   \newblock  \emph{Indiana Univ. Math. J.} \textbf{30} (1981), no.~5, 713--747.

\bibitem{HuZ1}
	\newblock Jiaxin Hu and Martina Z\"{a}hle
	\newblock Potential spaces on fractals,
	\newblock \emph{Studia Math.} \textbf{170} (2005), no.~3, 259--281.

\bibitem{HuZ2}
	\newblock Jiaxin Hu and Martina Z\"{a}hle
	\newblock Generalized {B}essel and {R}iesz potentials on metric measure spaces,
	\newblock \emph{Potential Anal.} \textbf{30} (2009), no.~4, 315--340.

\bibitem{IRT}
   \newblock Marius Ionescu, Luke~G. Rogers, and Alexander Teplyaev,
   \newblock Derivations and  {D}irichlet forms on fractals,
   \newblock   \emph{J. Funct. Anal.} \textbf{263} (2012), no.~8,  2141--2169.

\bibitem{KP}
   \newblock Tosio Kato and Gustavo Ponce,
   \newblock Commutator estimates and the Euler and Navier-Stokes equations,
   \newblock  \emph{Comm. Pure Appl. Math.} \textbf{41} (1988), no. 7, 891–907. 

\bibitem{Krebs} 
   \newblock     William B. Krebs,
   \newblock     A diffusion defined on a fractal state space,
   \newblock     \emph{Stochastic Process. Appl.}, \textbf{37} (1991) 199--212.

\bibitem{Kigamibook} 
   \newblock    Jun Kigami, 
   \newblock      \emph{Analysis on Fractals,}
   \newblock    Cambridge Tracts in Mathematics,  143, Cambridge,  2001.

\bibitem{KigamiJFA03} 
    \newblock  Jun  Kigami, 
    \newblock   Harmonic analysis for resistance forms,
    \newblock    \emph{J. Funct. Anal.}, \textbf{204} (2003), 399--444.

\bibitem{RST}
  \newblock  	Luke G. Rogers,  Robert S. Strichartz and Alexander  Teplyaev,
   \newblock   Smooth bumps, a Borel theorem and partitions of smooth functions on P.C.F. fractals,
  \newblock   	  \emph{Trans. Amer. Math. Soc.}, \textbf{361} (2009), 1765--1790.


\bibitem{Strich}
   \newblock Robert~S. Strichartz,
   \newblock Multipliers on fractional Sobolev spaces,
   \newblock  \emph{J. Math. Mech.} \textbf{16} (1967), 1031--1060.

\bibitem{S}
    \newblock  Robert S. Strichartz,
     \newblock  Function spaces on fractals,
   \newblock   \emph{J. Funct. Anal.},   \textbf{198} (2003), 43--83.


\bibitem{bobsbook}
      \newblock  Robert S. Strichartz,
     \newblock   \emph{Differential Equations on Fractals, a tutorial},
     \newblock    Princeton University Press, Princeton, NJ, 2006.

\end{thebibliography}
\end{document}